\documentclass[11pt,a4paper,reqno]{amsart}
\usepackage{fancyhdr}
\usepackage{ifpdf}
\usepackage{amsmath,amsfonts,amsthm,amsopn,amssymb}
\usepackage{verbatim,wasysym,mathrsfs}
\usepackage[left=2.6cm,right=2.6cm,top=3.1cm,bottom=3.1cm]{geometry}
\usepackage{cite,marginnote}
\usepackage{color,enumitem,graphicx}
\usepackage[colorlinks=true,urlcolor=blue,citecolor=red,linkcolor=blue,
linktocpage,pdfpagelabels,bookmarksnumbered,bookmarksopen]{hyperref}
\usepackage[english]{babel}
\usepackage[T1]{fontenc}
\usepackage{lmodern}
\usepackage[hyperpageref]{backref}

\newtheorem{Thm}{Theorem}[section]
\newtheorem{Lem}[Thm]{Lemma}

\newtheorem{Cor}[Thm]{Corollary}
\newtheorem{Prop}[Thm]{Proposition}

\theoremstyle{definition}

\newtheorem{Rem}{Remark}[section]

\newcommand{\R}{{\mathbb R}}
\newcommand{\weakto}{{\rightharpoonup}}

\newcommand{\dd}{\,\mathrm{d}}
\newcommand{\norm}[1]{\lVert#1\rVert}

\numberwithin{equation}{section}
\allowdisplaybreaks[2]

\title[Concentration levels and Choquard--Trudinger--Moser extremals]{Sharp concentration levels and extremal functions for a planar Choquard-type Trudinger--Moser inequality}

\author{Jinkai Gao}

\address[Jinkai Gao]{\newline\indent Department of Basic Courses
\newline\indent Beijing University of Agriculture
\newline\indent Beijing, 102206, China.}
    \email{\href{mailto:20269801@bua.edu.cn}{20269801@bua.edu.cn}}

\subjclass[2020]{35J20, 35B33, 46E35, 49J35}

\keywords{Trudinger--Moser inequality, Hardy--Littlewood--Sobolev inequality, Choquard functional, concentration--compactness, extremal function, Green function}

\date{\today}

\begin{document}
\begin{abstract}
Let $\Omega\subset\R^2$ be a smooth bounded domain and let $0<\alpha<2$.
We study the exponential Choquard functional
\[
 \mathcal F_\beta(u)=\int_\Omega\!\int_\Omega
 \frac{e^{\beta u^2(x)}e^{\beta u^2(y)}}{|x-y|^\alpha}\dd x\dd y
\]
on the unit Dirichlet ball of $H_0^1(\Omega)$. We establish the sharp
boundedness threshold $\beta_*=(4-\alpha)\pi$ and prove attainment of
the critical supremum. More precisely, we determine the exact
concentration level at every point of $\overline\Omega$ in terms of
the regular part of the Dirichlet Green function. The concentration
upper bound follows from the classical planar Trudinger--Moser
concentration formula and the sharp Hardy--Littlewood--Sobolev
inequality, after separating the background contribution. A test
function obtained by matching a logarithmic profile to the Green
function attains this bound asymptotically. Its exact Dirichlet
normalization produces a strictly positive lower-order correction,
which rules out concentration of maximizing sequences. We also prove
the Euler--Lagrange equation, positivity, and regularity of extremals.
\end{abstract}
\maketitle

\section{Introduction}\label{sec:intro}

Let $\Omega\subset\R^2$ be a bounded connected domain with smooth
boundary. Throughout the paper, the kernel exponent $\alpha\in(0,2)$
is fixed, all Sobolev functions are real-valued, and
$\norm{u}_{H_0^1(\Omega)}:=\norm{\nabla u}_{L^2(\Omega)}$.
We consider the exponential Riesz interaction
\begin{equation}\label{eq:main-functional}
 \mathcal F_\beta(u):=\int_\Omega\!\int_\Omega
 \frac{e^{\beta u^2(x)}e^{\beta u^2(y)}}{|x-y|^\alpha}\dd x\dd y,
 \qquad u\in H_0^1(\Omega),\quad\beta>0.
\end{equation}
Set
\begin{equation}\label{eq:critical-parameters}
 \beta_*:=(4-\alpha)\pi,\qquad
 q_\alpha:=\frac4{4-\alpha},\qquad q_\alpha\beta_*=4\pi,
\end{equation}
and write
\[
 \mathcal S(\Omega):=\{u\in H_0^1(\Omega):\norm{\nabla u}_2\le1\},
 \qquad
 \Lambda_0:=\sup_{u\in\mathcal S(\Omega)}\mathcal F_{\beta_*}(u).
\]
The subscript $0$ in $\Lambda_0$ denotes the critical supremum,
not the value corresponding to $\beta=0$.

The classical Trudinger--Moser inequality
\cite{Moser,Trudinger-1967} describes the limiting exponential
integrability of planar Sobolev functions. Carleson and Chang
\cite{CarlesonChang1986} proved attainment of the critical supremum
on balls, and Flucher \cite{Flucher1992} extended this result to
general bounded planar domains by relating concentration levels
to the Dirichlet Green function. Lin \cite{Lin1996} established the
corresponding higher-dimensional existence result. Adimurthi and
Sandeep \cite{AdimurthiSandeep2007} introduced singularly weighted
versions of the embedding. Extremals for such local weighted
inequalities were obtained by Csat\'o and Roy \cite{CsatoRoy2015}
and, for a strengthened inequality, by Yang and Zhu
\cite{YangZhu2017}. Wang and Yang \cite{WangYang2021} studied
compactness of extremals as the singular parameter tends to zero.
In these weighted problems, the singularity occurs in a factor
depending on one spatial variable; in \eqref{eq:main-functional},
it occurs in a kernel coupling two variables.

Critical exponential nonlinearities with Riesz convolution also
arise in planar Choquard equations. Alves, Cassani, Tarsi, and Yang
\cite{AlvesCassaniTarsiYang2016} established existence and
concentration of ground state solutions for a critical nonlocal
Schr\"odinger equation in $\R^2$. On bounded domains, Rawat,
R\u{a}dulescu, and Sreenadh \cite{RawatRadulescuSreenadh2024}
studied positive solutions with critical exponential interactions
both in the equation and in a nonlinear Neumann boundary
condition. Their Lemma~2.3(i) provides a threshold estimate for an
exponential Riesz double integral on an $H^1(\Omega)$ unit ball,
with a norm involving Dirichlet, lower-order, and boundary terms.
This is a related boundedness problem, but its boundary conditions
and constraint differ from those considered here. Existence of
solutions to such equations does not by itself establish
maximization of \eqref{eq:main-functional} on $\mathcal S(\Omega)$.

Another closely related line of research concerns logarithmic
convolution kernels. Cingolani and Weth \cite{CingolaniWeth2022}
studied the associated Trudinger--Moser-type maximization problems
on balls and on the whole plane. Cingolani, Weth, and Yu
\cite{CingolaniWethYu2024} analyzed the borderline growth regime,
computed concentration thresholds, and obtained sufficient
conditions for attainment. Higher-dimensional logarithmic-kernel
inequalities were studied by Cannone and Cingolani
\cite{CannoneCingolani2025}; borderline concentration estimates and
extremals on balls are also considered in the preprint of Cannone
and Yu \cite{CannoneYu2026}. The kernel in
\eqref{eq:main-functional} is instead positive and homogeneous,
and our domain is not required to be a disk. We do not assume
radial symmetry of maximizing sequences.

For the present Dirichlet constraint, the boundedness threshold
$\beta_*$ follows by combining the Hardy--Littlewood--Sobolev
(HLS) inequality with the classical Trudinger--Moser inequality,
and its optimality follows from a Moser sequence. The principal
issue is the behavior at the critical threshold: weak convergence
does not in general preserve the nonlocal interaction. Our aim is
to determine the exact concentration level of
$\mathcal F_{\beta_*}$ and to prove that it lies strictly below the
variational supremum. This strict comparison gives attainment and
strong compactness of normalized maximizing sequences.

Our first theorem collects the sharp threshold, attainment, and
the equation and regularity of extremals.
\begin{Thm}[Sharp threshold and attainment]\label{thm:main}
Let $\Omega\subset\R^2$ be a bounded connected smooth domain and
$0<\alpha<2$. Then
\[
 \sup_{u\in\mathcal S(\Omega)}\mathcal F_\beta(u)
 \begin{cases}
 <\infty,&0<\beta\le\beta_*,\\
 =\infty,&\beta>\beta_*.
 \end{cases}
\]
For every $0<\beta\le\beta_*$, the finite supremum is attained by
some nonnegative function $u_\beta\in H_0^1(\Omega)$ with
$\norm{\nabla u_\beta}_2=1$. Every nonnegative maximizer $u_\beta$
is positive in $\Omega$ and satisfies, in the weak sense,
\[
 \begin{cases}
 \displaystyle
 -\Delta u_\beta(x)=\frac1{\lambda_\beta}
 \left(\int_\Omega
 \frac{e^{\beta u_\beta^2(y)}}{|x-y|^\alpha}\dd y\right)
 u_\beta(x)e^{\beta u_\beta^2(x)}
 &\text{in }\Omega,\\[7pt]
 \quad\ \  u_\beta=0&\text{on }\partial\Omega,
 \end{cases}
\]
where
\[
 \lambda_\beta:=\int_\Omega\!\int_\Omega
 \frac{u_\beta^2(x)e^{\beta u_\beta^2(x)}
 e^{\beta u_\beta^2(y)}}{|x-y|^\alpha}\dd x\dd y
 \in(0,\infty).
\]
Every such maximizer belongs to $W^{2,p}(\Omega)$ for every
finite $p>1$, and hence to $C^{1,\gamma}(\overline\Omega)$ for every
$0<\gamma<1$. In particular, there exists $u_*:=u_{\beta_*}$ with
$\norm{\nabla u_*}_2=1$ and $\mathcal F_{\beta_*}(u_*)=\Lambda_0$.
\end{Thm}

Our second theorem quantifies the loss of compactness at the
critical exponent. A sequence $(u_n)\subset\mathcal S(\Omega)$
is said to \emph{concentrate at $x_0\in\overline\Omega$} if
\begin{equation}\label{eq:concentration-definition}
 |\nabla u_n|^2\dd x\weakto\delta_{x_0}
 \quad\text{in }\mathcal M(\overline\Omega),
\end{equation}
where convergence is tested against continuous functions on
$\overline\Omega$. In particular, its Dirichlet energies tend to
one. Define
\begin{equation}\label{eq:concentration-level-definition}
 \mathcal C_\alpha(x_0):=
 \sup\left\{\limsup_{n\to\infty}\mathcal F_{\beta_*}(u_n):
 (u_n)\subset\mathcal S(\Omega)\text{ satisfies }
 \eqref{eq:concentration-definition}\right\}.
\end{equation}
The background value is
\[
 I_\alpha(\Omega):=\int_\Omega\!\int_\Omega|x-y|^{-\alpha}\dd x\dd y
 =\mathcal F_{\beta_*}(0).
\]
For $x_0\in\Omega$, let $G(\cdot,x_0)$ be the Dirichlet Green
function of $-\Delta$ and write, near its pole,
\begin{equation}\label{eq:Green-expansion}
 G(x,x_0)=-\frac1{2\pi}\log|x-x_0|+A_{x_0}+\eta_{x_0}(x),
 \qquad \eta_{x_0}(x_0)=0,
\end{equation}
where $\eta_{x_0}$ is smooth and harmonic. In the convention
$G(x,y)=-(2\pi)^{-1}\log|x-y|-H(x,y)$, used in
\cite{Flucher1992,CsatoRoy2015}, one has $A_{x_0}=-H(x_0,x_0)$.
Thus the conformal radius is
\begin{equation}\label{eq:conformal-radius}
 r_\Omega(x_0):=e^{2\pi A_{x_0}}\quad(x_0\in\Omega),
 \qquad r_\Omega(x_0):=0\quad(x_0\in\partial\Omega).
\end{equation}
The boundary value is the continuous extension of the interior
definition. The same Green-function definition is used on
multiply connected domains; see \cite[Section~4.2]{Flucher1992}.
Set
\begin{equation}\label{eq:K-alpha}
 K_\alpha(x_0):=\frac{2\pi^2}{2-\alpha}
 e^{(4-\alpha)/2}r_\Omega(x_0)^{4-\alpha},
 \qquad x_0\in\overline\Omega.
\end{equation}

\begin{Thm}[Exact concentration level]\label{thm:concentration}
Let $\Omega\subset\R^2$ be a bounded connected smooth domain and
$0<\alpha<2$. Then
\begin{equation}\label{eq:concentration-formula}
 \mathcal C_\alpha(x_0)=I_\alpha(\Omega)+K_\alpha(x_0)
 \qquad\text{for every }x_0\in\overline\Omega.
\end{equation}
For each $x_0\in\Omega$, there exists a sequence
$(W_n)\subset\mathcal S(\Omega)$, with
$\norm{\nabla W_n}_2=1$, concentrating at $x_0$ such that
\[
 \mathcal F_{\beta_*}(W_n)\longrightarrow
 I_\alpha(\Omega)+K_\alpha(x_0),\qquad
 \mathcal F_{\beta_*}(W_n)>I_\alpha(\Omega)+K_\alpha(x_0)
\]
for all sufficiently large $n$. Moreover,
\begin{equation}\label{eq:global-gap}
 \Lambda_0>\max_{x\in\overline\Omega}\mathcal C_\alpha(x).
\end{equation}
\end{Thm}

Formula~\eqref{eq:concentration-formula} separates the constant
background from the excess interaction created by concentration.
It applies to arbitrary concentrating sequences, not only to
sequences satisfying an Euler--Lagrange equation. In particular,
concentration at a boundary point leaves only the background value
$I_\alpha(\Omega)$. The role of the domain geometry is encoded by
$r_\Omega(x_0)^{4-\alpha}$, and the test sequence proves that the
coefficient in this formula cannot be decreased.

The critical attainment assertion in Theorem~\ref{thm:main}
follows from Theorem~\ref{thm:concentration} and the local
concentration--compactness alternative. We therefore prove
Theorem~\ref{thm:concentration} before completing
Theorem~\ref{thm:main}. The argument follows the classical strategy
of comparing the supremum with concentration levels
\cite{Flucher1992,Lin1996,CsatoRoy2015}, with two points specific
to the present interaction. First, the background must be removed
before applying the sharp HLS inequality. For a weakly vanishing
sequence in the unit Dirichlet ball, the increment
$e^{\beta_*u_n^2}-1$ tends to zero in $L^1(\Omega)$, so its mixed
interaction with the constant density vanishes. Its nonlocal
quadratic interaction need not vanish. The sharp HLS inequality
bounds that term by a power of the local increment
$\int_\Omega(e^{4\pi u_n^2}-1)\dd x$. Flucher's local concentration
formula then yields the upper bound with the precise constant in
\eqref{eq:K-alpha}.

Second, a test sequence matching the upper bound at leading order
is not enough to prove the strict gap. We match a logarithmic
profile to the Dirichlet Green function so that the leading
rescaled inner density has the shape of an HLS equality function.
Exact Dirichlet normalization produces a strictly positive
lower-order correction, and the scales are chosen so that all
matching and truncation errors are of smaller order.
This establishes both the sharpness of the concentration bound
and \eqref{eq:global-gap}. Concentration of maximizing sequences
is thereby excluded, and local compactness transfers to the
nonlocal functional through the HLS inequality. No classification
of entire solutions of a nonlocal equation is needed.

The paper is organized as follows.
Section~\ref{sec:prelim} collects the analytic preliminaries,
including the classical local concentration results.
Section~\ref{sec:sharp-compactness} proves the sharp exponential
threshold, the transfer of local compactness, and the background
splitting. Section~\ref{sec:concentration-upper} establishes the
upper bound for arbitrary concentrating sequences.
Section~\ref{sec:test} constructs the test sequence and proves
Theorem~\ref{thm:concentration}. Finally,
Section~\ref{sec:attainment} proves attainment and strong
compactness of normalized maximizing sequences, and derives
the Euler--Lagrange equation, positivity, and regularity.

\section{Preliminaries}\label{sec:prelim}

We write $\norm{\cdot}_p$ for the norm in $L^p(\Omega)$
unless another domain is specified.
For $r>0$ and $x\in\R^2$, let $B_r(x)$ denote the open
ball of radius $r$ centered at $x$, and write $B_r:=B_r(0)$.
The notation $|E|$ denotes the Lebesgue measure of a
measurable set $E$, and $\mathbf1_E$ denotes its
characteristic function.
Constants denoted by $C$ may change from line to line.
They may depend on the fixed exponent $\alpha$ and on
$\Omega$; any additional dependence is indicated as needed.
No estimate is asserted to be uniform as $\alpha$
approaches an endpoint of $(0,2)$.

\subsection{The sharp Hardy--Littlewood--Sobolev inequality}

For $D=\Omega$ or $\R^2$, define
\[
 \mathcal B_{\alpha,D}(f,g):=
 \int_D\!\int_D\frac{f(x)g(y)}{|x-y|^\alpha}\dd x\dd y,
 \qquad
 \mathcal I_{\alpha,D}[f](x):=
 \int_D\frac{f(y)}{|x-y|^\alpha}\dd y.
\]
We omit $D$ only when the domain is unambiguous. When applying a
whole-space inequality to a density $h$ on $\Omega$, its zero
extension is $\mathbf1_\Omega h$. In particular, the zero extension
of $e^{\beta u^2}$ is $\mathbf1_\Omega e^{\beta u^2}$, not the
exponential of the zero extension of $u$.

\begin{Lem}[Sharp planar HLS inequality]\label{lem:hls}
\textup{\cite{Lieb1983,Lieb2001}}
For $q_\alpha$ as in \eqref{eq:critical-parameters} and
$f,g\in L^{q_\alpha}(\R^2)$,
\begin{equation}\label{eq:hls}
 |\mathcal B_{\alpha,\R^2}(f,g)|
 \le C_{2,\alpha}\norm{f}_{L^{q_\alpha}(\R^2)}
 \norm{g}_{L^{q_\alpha}(\R^2)},
 \qquad
 C_{2,\alpha}=\frac{2\pi^{\alpha/2}}{2-\alpha}.
\end{equation}
The constant is optimal. For nonzero nonnegative $f,g$, equality
holds precisely when
\[
 f(x)=a(\gamma^2+|x-\xi|^2)^{-(4-\alpha)/2},\qquad
 g(x)=b(\gamma^2+|x-\xi|^2)^{-(4-\alpha)/2},
\]
for some $a,b,\gamma>0$ and $\xi\in\R^2$.
\end{Lem}

The same estimate with $D=\Omega$ follows by zero extension.
Since $0<\alpha<2$ and $\Omega$ is bounded, we also have
\begin{equation}\label{eq:background-potential}
 M_{\alpha,\Omega}:=
 \sup_{y\in\Omega}\int_\Omega|x-y|^{-\alpha}\dd x<\infty.
\end{equation}
Indeed, after translating $y$ to the origin, the domain is contained
in $B_{\operatorname{diam}\Omega}(0)$, and polar integration
bounds the integral by
$2\pi(\operatorname{diam}\Omega)^{2-\alpha}/(2-\alpha)$.

\subsection{Exponential integrability}

\begin{Lem}[Trudinger--Moser inequality]\label{lem:mt}
\textup{\cite{Moser,Trudinger-1967}}
There exists an absolute constant $C_{\mathrm{MT}}>0$ such that
for every bounded open set $D\subset\R^2$,
\begin{equation}\label{eq:mt}
 \sup_{\substack{v\in H_0^1(D)\\\norm{\nabla v}_{L^2(D)}\le1}}
 \int_D e^{4\pi v^2}\dd x\le C_{\mathrm{MT}}|D|.
\end{equation}
The exponent $4\pi$ is sharp. The same upper bound holds with
$4\pi$ replaced by any smaller positive coefficient.
\end{Lem}

The following elementary consequence will justify nonlinear
integrability and differentiation at a fixed Sobolev function.
It does not assert uniform bounds on the full unit ball above the
critical exponent.
\begin{Lem}[Local bounds around a fixed function]\label{lem:fixed-function}
For every $u\in H_0^1(\Omega)$ and $b>0$, there exist
$d=d(u,b)>0$ and $C=C(u,b)>0$ such that
\begin{equation}\label{eq:fixed-neighborhood}
 \sup_{\norm{\nabla(v-u)}_2\le d}
 \int_\Omega e^{b v^2}\dd x\le C.
\end{equation}
In particular, $e^{b u^2}\in L^1(\Omega)$ for every fixed $b>0$.
\end{Lem}
\begin{proof}

For $M>0$, define $u_M:=\max\{-M,\min\{u,M\}\}$ and $z_M:=u-u_M$.
Then $u_M,z_M\in H_0^1(\Omega)$, $|u_M|\le M$, and
\[
 \norm{\nabla z_M}_2^2
 =\int_{\{|u|>M\}}|\nabla u|^2\dd x\longrightarrow0\qquad\text{as }M\to+\infty.
\]
Choose $d>0$ so small that $2b(2d)^2\le4\pi$, and then choose
$M>0$ such that $\norm{\nabla z_M}_2\le d$.
If $\norm{\nabla(v-u)}_2\le d$, the function
$w:=z_M+v-u$ satisfies $\norm{\nabla w}_2\le2d$, and
$v=u_M+w$. Hence
\[
 \int_\Omega e^{b v^2}\dd x
 \le e^{2bM^2}\int_\Omega e^{2b w^2}\dd x
 \le e^{2bM^2}C_{\mathrm{MT}}|\Omega|
\]
by \eqref{eq:mt}. This proves \eqref{eq:fixed-neighborhood}.
\end{proof}

\subsection{Local concentration and Green functions}

The next lemma concerns only the \emph{local} exponential density.
The corresponding alternative for the nonlocal functional will be
derived separately.
\begin{Lem}[Local concentration--compactness alternative]
\label{lem:local-alternative}
\textup{\cite[Theorem~1]{Flucher1992}; see also
\cite{Lions-1985-REV} and \cite[Theorem~6]{CsatoRoy2015}}
For every sequence $(u_n)\subset\mathcal S(\Omega)$, a subsequence
converges weakly to $u\in\mathcal S(\Omega)$ and almost everywhere
in $\Omega$, and at least one of the following conclusions holds:
\begin{enumerate}[label=\textup{(\roman*)},leftmargin=*]
\item $(u_n)$ concentrates at some $x_0\in\overline\Omega$ and $u=0$;
\item $e^{4\pi u_n^2}\to e^{4\pi u^2}$ strongly in $L^1(\Omega)$.
\end{enumerate}
The two conclusions are not asserted to be mutually exclusive.
\end{Lem}

The regular part in \eqref{eq:Green-expansion} has the opposite sign
to the function denoted by $H$ in \cite{Flucher1992,CsatoRoy2015}:
$A_{x_0}=-H_{\Omega,x_0}(x_0)$. Thus
\eqref{eq:conformal-radius} agrees with their Green-function
normalization. The function $r_\Omega$ is continuous on
$\overline\Omega$, positive in $\Omega$, and zero on
$\partial\Omega$; see \cite[Proposition~12]{Flucher1992}.

\begin{Lem}[Classical local concentration bound]\label{lem:Flucher}
\textup{\cite[Lemma~4, Theorem~13, and Corollary~15]{Flucher1992}}
If $(u_n)\subset\mathcal S(\Omega)$ concentrates at
$x_0\in\overline\Omega$, then
\begin{equation}\label{eq:local-concentration-bound}
 \limsup_{n\to\infty}
 \int_\Omega(e^{4\pi u_n^2}-1)\dd x
 \le \pi e\,r_\Omega(x_0)^2.
\end{equation}
In particular, the integral on the left tends to zero when
$x_0\in\partial\Omega$.
\end{Lem}

The coefficient in \eqref{eq:local-concentration-bound} is the
Carleson--Chang concentration value on the disk, transported by
Flucher's concentration formula. We invoke that local result as a
known theorem. For a detailed discussion of the sequence
modifications used in the domain-to-disk argument, see
\cite[Section~6, especially Remarks~28 and~30]{CsatoRoy2015}.
Neither of the two preceding lemmas assumes that $(u_n)$ solves an
Euler--Lagrange equation.

\section{The sharp inequality and nonlocal compactness}
\label{sec:sharp-compactness}

\subsection{The sharp exponential threshold}

By Lemma~\ref{lem:fixed-function} and \eqref{eq:hls},
$\mathcal F_\beta(u)<\infty$ for every individual
$u\in H_0^1(\Omega)$ and every $\beta>0$. The next proposition
concerns uniform boundedness on $\mathcal S(\Omega)$.
\begin{Prop}[Sharp boundedness threshold]\label{prop:sharp}
The boundedness dichotomy in Theorem~\ref{thm:main} holds.
\end{Prop}
\begin{proof}
If $0<\beta\le\beta_*$ and $u\in\mathcal S(\Omega)$, apply
\eqref{eq:hls} to $\mathbf1_\Omega e^{\beta u^2}$ to obtain
\[
 \begin{aligned}
 \mathcal F_\beta(u)
 &\le C_{2,\alpha}
 \left(\int_\Omega e^{q_\alpha\beta u^2}\dd x\right)^{2/q_\alpha}\\
 &\le C_{2,\alpha}
 \bigl(C_{\mathrm{MT}}|\Omega|\bigr)^{(4-\alpha)/2}.
 \end{aligned}
\]
Here $q_\alpha\beta\le4\pi$ by \eqref{eq:critical-parameters},
so \eqref{eq:mt} applies, including at $\beta=\beta_*$.

For $\beta>\beta_*$, translate coordinates so that $0\in\Omega$
and fix $L>0$ with $B_L(0)\Subset\Omega$. For $0<\ell<L$, set
\[
 m_\ell(x)=\frac1{\sqrt{2\pi}}
 \begin{cases}
 \sqrt{\log(L/\ell)},&|x|\le\ell,\\[2pt]
 \displaystyle\frac{\log(L/|x|)}{\sqrt{\log(L/\ell)}},
 &\ell<|x|<L,\\[5pt]
 0,&|x|\ge L.
 \end{cases}
\]
Then $m_\ell\in H_0^1(\Omega)$ and
\[
 \int_\Omega|\nabla m_\ell|^2\dd x
 =\frac1{\log(L/\ell)}\int_\ell^L\frac{\dd r}{r}=1.
\]
Since $e^{\beta m_\ell^2}=(L/\ell)^{\beta/(2\pi)}$ on $B_\ell$,
restriction to $B_\ell\times B_\ell$ and a change of variables give
\[
 \mathcal F_\beta(m_\ell)
 \ge L^{\beta/\pi}\ell^{4-\alpha-\beta/\pi}
 \int_{B_1}\!\int_{B_1}|z-w|^{-\alpha}\dd z\dd w.
\]
The last integral is finite and positive. Because
$4-\alpha-\beta/\pi<0$, the lower bound tends to infinity as
$\ell\downarrow0$.
\end{proof}

Taking any nonzero function in $\mathcal S(\Omega)$ yields
\begin{equation}\label{eq:strict-background}
 I_\alpha(\Omega)<\Lambda_0<\infty,
\end{equation}
since the kernel is positive and the exponential product is strictly
larger than one on a set of positive measure.

\subsection{Continuity and separation of the background}

\begin{Lem}[Transfer of local compactness]\label{lem:transfer}
Suppose $u_n,u\in H_0^1(\Omega)$ and
$e^{4\pi u_n^2}\to e^{4\pi u^2}$ in $L^1(\Omega)$. Then
\begin{equation}\label{eq:nonlocal-density-convergence}
 e^{\beta_*u_n^2}\longrightarrow e^{\beta_*u^2}
 \quad\text{in }L^{q_\alpha}(\Omega),
 \qquad
 \mathcal F_{\beta_*}(u_n)\longrightarrow\mathcal F_{\beta_*}(u).
\end{equation}
\end{Lem}
\begin{proof}
For $A,B\ge0$ and $q\ge1$, one has
$|A-B|^q\le|A^q-B^q|$. Hence
\[
 \norm{e^{\beta_*u_n^2}-e^{\beta_*u^2}}_{q_\alpha}^{q_\alpha}
 \le\norm{e^{4\pi u_n^2}-e^{4\pi u^2}}_1\longrightarrow0.
\]
If $h_n=e^{\beta_*u_n^2}$ and $h=e^{\beta_*u^2}$, bilinearity
and \eqref{eq:hls} imply
\[
 |\mathcal B_{\alpha,\Omega}(h_n,h_n)
   -\mathcal B_{\alpha,\Omega}(h,h)|
 \le C_{2,\alpha}\norm{h_n-h}_{q_\alpha}
       (\norm{h_n}_{q_\alpha}+\norm{h}_{q_\alpha})\longrightarrow0.
\]
\end{proof}

\begin{Lem}[Background splitting]\label{lem:background-splitting}
Let $(u_n)\subset\mathcal S(\Omega)$ satisfy
$u_n\weakto0$ in $H_0^1(\Omega)$, and set
\[
 v_n:=e^{\beta_*u_n^2}-1,\qquad
 J(u_n):=\int_\Omega(e^{4\pi u_n^2}-1)\dd x.
\]
Then
\begin{equation}\label{eq:background-splitting}
 \mathcal F_{\beta_*}(u_n)
 =I_\alpha(\Omega)+\mathcal B_{\alpha,\Omega}(v_n,v_n)+o_n(1),
\end{equation}
and
\begin{equation}\label{eq:local-to-nonlocal-bound}
 \mathcal F_{\beta_*}(u_n)
 \le I_\alpha(\Omega)+C_{2,\alpha}J(u_n)^{(4-\alpha)/2}+o_n(1).
\end{equation}
\end{Lem}
\begin{proof}
Because $\alpha>0$, we have $\beta_*<4\pi$. Choose $s>1$ such
that $s\beta_*<4\pi$. The Trudinger--Moser inequality gives
$\sup_n\norm{e^{\beta_*u_n^2}}_s<\infty$.
Compact Sobolev embedding implies $u_n\to0$ in measure.
Consequently, Vitali's theorem yields $v_n\to0$ in $L^1(\Omega)$.
In particular, by \eqref{eq:background-potential},
\[
 0\le\mathcal B_{\alpha,\Omega}(1,v_n)
 \le M_{\alpha,\Omega}\norm{v_n}_1\longrightarrow0.
\]
Expanding $\mathcal B_{\alpha,\Omega}(1+v_n,1+v_n)$ proves
\eqref{eq:background-splitting}.

The inequality $(t-1)^q\le t^q-1$ for $t\ge1$, $q\ge1$,
together with \eqref{eq:critical-parameters}, gives
\[
 \norm{v_n}_{q_\alpha}^{q_\alpha}
 \le\int_\Omega(e^{q_\alpha\beta_*u_n^2}-1)\dd x=J(u_n).
\]
Applying \eqref{eq:hls} to the quadratic term in
\eqref{eq:background-splitting} proves
\eqref{eq:local-to-nonlocal-bound}.
\end{proof}

\begin{Cor}[Nonlocal concentration--compactness alternative]
\label{cor:nonlocal-alternative}
Every sequence $(u_n)\subset\mathcal S(\Omega)$ has a weakly
convergent subsequence $u_n\weakto u$ for which at least one of the
following holds: either $u_n$ concentrates at a point of
$\overline\Omega$, or \eqref{eq:nonlocal-density-convergence} holds.
\end{Cor}
\begin{proof}
Apply Lemma~\ref{lem:local-alternative} to the local density and
then Lemma~\ref{lem:transfer} to its compactness conclusion.
\end{proof}

\section{The concentration upper bound}\label{sec:concentration-upper}

\begin{Prop}[Upper bound for arbitrary concentrating sequences]
\label{prop:concentration-upper}
If $(u_n)\subset\mathcal S(\Omega)$ concentrates at
$x_0\in\overline\Omega$, then
\begin{equation}\label{eq:concentration-upper}
 \limsup_{n\to\infty}\mathcal F_{\beta_*}(u_n)
 \le I_\alpha(\Omega)+K_\alpha(x_0).
\end{equation}
If $x_0\in\partial\Omega$, then
$\mathcal F_{\beta_*}(u_n)\to I_\alpha(\Omega)$.
\end{Prop}
\begin{proof}
A concentrating sequence converges weakly to zero in
$H_0^1(\Omega)$. For completeness, for every smooth test function
$\zeta$ and every fixed $r>0$, split
$\int_\Omega\nabla u_n\cdot\nabla\zeta$ into $B_r(x_0)$ and
its complement. The first part is bounded by
$\norm{\nabla\zeta}_{L^2(\Omega\cap B_r(x_0))}$, while the second
tends to zero by concentration. Sending $n\to\infty$ and then
$r\downarrow0$ identifies every weak limit as zero.

Now combine \eqref{eq:local-to-nonlocal-bound} with
\eqref{eq:local-concentration-bound} to obtain
\[
 \limsup_{n\to\infty}\mathcal F_{\beta_*}(u_n)
 \le I_\alpha(\Omega)+
 C_{2,\alpha}\bigl(\pi e\,r_\Omega(x_0)^2\bigr)^{(4-\alpha)/2}.
\]
The constants match exactly:
\begin{equation}\label{eq:constant-matching}
 \begin{aligned}
 C_{2,\alpha}\bigl(\pi e\,r_\Omega(x_0)^2\bigr)^{(4-\alpha)/2}
 &=\frac{2\pi^2}{2-\alpha}
   e^{(4-\alpha)/2}r_\Omega(x_0)^{4-\alpha}\\
 &=K_\alpha(x_0).
 \end{aligned}
\end{equation}
This proves \eqref{eq:concentration-upper}. At a boundary point,
$r_\Omega(x_0)=0$ and the lower bound
$\mathcal F_{\beta_*}(u_n)\ge I_\alpha(\Omega)$ gives convergence
to the background value.
\end{proof}

\begin{Rem}\label{rem:upper-scope}
The estimate \eqref{eq:concentration-upper} concerns arbitrary
concentrating sequences, not merely subcritical extremals. Its proof
uses the local concentration theorem only after subtracting the
background. Applying HLS directly to $e^{\beta_*u_n^2}$ would
retain the background inside the power and would not give
\eqref{eq:constant-matching}. Sharpness of the latter bound is a
separate issue, proved in the next section.
\end{Rem}

\section{Test functions and the exact concentration level}\label{sec:test}

\subsection{A normalized HLS profile}

The following explicit function is used as a trial profile. Its
normalization is obtained from the HLS equality case, independently
of any classification theorem for nonlocal differential equations.
\begin{Lem}[Profile identities]\label{lem:profile}
Set
\begin{equation}\label{eq:trial-profile}
 \kappa_\alpha:=\left(\frac{2\pi^2}{2-\alpha}\right)^{2/(4-\alpha)},
 \qquad
 \varphi(z):=-\frac1{4\pi}\log(1+\kappa_\alpha|z|^2),
 \qquad f(z):=e^{2\beta_*\varphi(z)}.
\end{equation}
Then
\begin{equation}\label{eq:profile-normalization}
 f(z)=(1+\kappa_\alpha|z|^2)^{-(4-\alpha)/2},\qquad
 \int_{\R^2}f^{q_\alpha}\dd z=\frac\pi{\kappa_\alpha},\qquad
 \mathcal B_{\alpha,\R^2}(f,f)=1.
\end{equation}
Moreover, for $R>0$,
\begin{equation}\label{eq:profile-energy}
 E_\varphi(R):=\int_{B_R}|\nabla\varphi|^2\dd z
 =\frac1{4\pi}\left[
 \log(1+\kappa_\alpha R^2)
 +\frac1{1+\kappa_\alpha R^2}-1\right].
\end{equation}
If
\[
 \theta_R:=\frac{\kappa_\alpha R^2}{4\pi(1+\kappa_\alpha R^2)},
\]
then
\begin{equation}\label{eq:profile-energy-defect}
 E_\varphi(R)+\varphi(R)=-\theta_R,
 \qquad \theta_R=\frac1{4\pi}+O(R^{-2})\quad(R\to\infty).
\end{equation}
Finally, for $R\ge1$,
\begin{equation}\label{eq:profile-tail}
 0\le1-\mathcal B_{\alpha,\R^2}
 (f\mathbf1_{B_R},f\mathbf1_{B_R})
 \le C R^{-(4-\alpha)/2}.
\end{equation}
\end{Lem}
\begin{proof}
The expression for $f$ follows from
\eqref{eq:critical-parameters}. Since
$f^{q_\alpha}(z)=(1+\kappa_\alpha|z|^2)^{-2}$, polar integration
gives $\int f^{q_\alpha}=\pi/\kappa_\alpha$.
The function $f$ is an HLS equality function. Consequently,
\[
 \mathcal B_{\alpha,\R^2}(f,f)
 =C_{2,\alpha}\left(\frac\pi{\kappa_\alpha}\right)^{2/q_\alpha}
 =\frac{2\pi^2}{2-\alpha}\kappa_\alpha^{-(4-\alpha)/2}=1.
\]
Next,
\[
 \varphi'(r)=-\frac{\kappa_\alpha r}
 {2\pi(1+\kappa_\alpha r^2)}.
\]
Integrating $2\pi\int_0^R|\varphi'(r)|^2r\dd r$ proves
\eqref{eq:profile-energy}; adding $\varphi(R)$ yields
\eqref{eq:profile-energy-defect}.
Finally,
\[
 \int_{\R^2\setminus B_R}f^{q_\alpha}\dd z
 =\frac{\pi}{\kappa_\alpha(1+\kappa_\alpha R^2)}.
\]
Thus $\norm{f\mathbf1_{\R^2\setminus B_R}}_{q_\alpha}
\le CR^{-2/q_\alpha}$. Split $f$ into its inner and outer parts
in \eqref{eq:profile-normalization} and apply \eqref{eq:hls}
to get \eqref{eq:profile-tail}.
\end{proof}

\subsection{Green-function matching and a positive correction}

\begin{Prop}[A concentrating test sequence above the threshold]
\label{prop:test}
For every $x_0\in\Omega$, there are functions
$W_\tau\in H_0^1(\Omega)$, defined for sufficiently small $\tau>0$,
with $\norm{\nabla W_\tau}_2=1$, concentrating at $x_0$ as
$\tau\downarrow0$. They satisfy
\begin{equation}\label{eq:test-limit}
 \mathcal F_{\beta_*}(W_\tau)\longrightarrow
 I_\alpha(\Omega)+K_\alpha(x_0),
\end{equation}
and, for some $c_{x_0}>0$,
\begin{equation}\label{eq:test-strict}
 \mathcal F_{\beta_*}(W_\tau)
 \ge I_\alpha(\Omega)+K_\alpha(x_0)
       +\frac{c_{x_0}}{\log(1/\tau)}
\end{equation}
for all sufficiently small $\tau$.
\end{Prop}
\begin{proof}
Fix $x_0\in\Omega$ and choose $\delta_0>0$ such that
$B_{\delta_0}(x_0)\Subset\Omega$ and the decomposition
\eqref{eq:Green-expansion} holds there. In this proof abbreviate
$A=A_{x_0}$, $\eta=\eta_{x_0}$, and $G=G(\cdot,x_0)$.

\smallskip\noindent
\emph{Step 1: matching the two pieces.}
For $\tau>0$ small, set
\begin{equation}\label{eq:test-scales}
 T:=\log(1/\tau),\qquad R:=T^2,\qquad \rho:=\tau R,
 \qquad
 Q_\tau:=\frac{T}{2\pi}+A+
 \frac1{4\pi}\log\frac{1+\kappa_\alpha R^2}{R^2}.
\end{equation}
Then $Q_\tau\sim T/(2\pi)$, $\rho\downarrow0$, and
$B_\rho(x_0)\Subset\Omega$ for $\tau$ small. Put
$c_\tau=\sqrt{Q_\tau}$ and define
\begin{equation}\label{eq:test-v}
 v_\tau(x)=
 \begin{cases}
 \displaystyle c_\tau+
 \frac{\varphi((x-x_0)/\tau)+\eta(x)}{c_\tau},
 &|x-x_0|<\rho,\\[7pt]
 \displaystyle \frac{G(x)}{c_\tau},&|x-x_0|\ge\rho.
 \end{cases}
\end{equation}
By \eqref{eq:test-scales},
\begin{equation}\label{eq:test-matching}
 Q_\tau+\varphi(R)=-\frac1{2\pi}\log\rho+A.
\end{equation}
The two traces in \eqref{eq:test-v} therefore agree exactly on
$\partial B_\rho(x_0)$. Since the outer piece has zero trace on
$\partial\Omega$, it follows that $v_\tau\in H_0^1(\Omega)$.

\smallskip\noindent
\emph{Step 2: exact Dirichlet normalization.}
Write $\varphi_\tau(x):=\varphi((x-x_0)/\tau)$ in the inner ball.
The gradient in the following calculation is with respect to $x$.
Since $\Delta\eta=0$ and $\varphi_\tau=\varphi(R)$ on the sphere,
\[
 \int_{B_\rho(x_0)}\nabla\varphi_\tau\cdot\nabla\eta\dd x
 =\varphi(R)\int_{\partial B_\rho(x_0)}\partial_r\eta\dd S=0.
\]
For the exterior energy, the outward normal of
$\Omega\setminus\overline{B_\rho(x_0)}$ on the inner sphere is
$-\partial_r$. Therefore Green's identity gives
\[
 \int_{\Omega\setminus B_\rho(x_0)}|\nabla G|^2\dd x
 =-\int_{\partial B_\rho(x_0)}G\,\partial_rG\dd S.
\]
Set $s_\rho=-(2\pi)^{-1}\log\rho+A$.
On that sphere, $G=s_\rho+\eta$ and
$\partial_rG=-(2\pi\rho)^{-1}+\partial_r\eta$.
The mean-value property and $\eta(x_0)=0$ imply
\[
 \int_{\partial B_\rho(x_0)}\eta\dd S=0,
 \qquad
 \int_{\partial B_\rho(x_0)}\partial_r\eta\dd S=0,
 \qquad
 \int_{\partial B_\rho(x_0)}\eta\partial_r\eta\dd S
 =\int_{B_\rho(x_0)}|\nabla\eta|^2\dd x.
\]
Consequently,
\begin{equation}\label{eq:exact-exterior-energy}
 \int_{\Omega\setminus B_\rho(x_0)}|\nabla G|^2\dd x
 =s_\rho-\int_{B_\rho(x_0)}|\nabla\eta|^2\dd x.
\end{equation}
The inner and outer $\eta$-energies cancel. Using
\eqref{eq:test-matching}, \eqref{eq:exact-exterior-energy}, and
\eqref{eq:profile-energy-defect}, we obtain the exact identity
\begin{equation}\label{eq:exact-normalization}
 \norm{\nabla v_\tau}_2^2
 =\frac{E_\varphi(R)+s_\rho}{Q_\tau}
 =1+\frac{E_\varphi(R)+\varphi(R)}{Q_\tau}
 =1-\frac{\theta_R}{Q_\tau}.
\end{equation}
For $\tau$ small, $Q_\tau>\theta_R$. Define
\begin{equation}\label{eq:normalized-test}
 W_\tau:=\frac{v_\tau}{\sqrt{1-\theta_R/Q_\tau}}.
\end{equation}
Then \eqref{eq:exact-normalization} gives
$\norm{\nabla W_\tau}_2=1$.
For every fixed $\delta>0$ with $B_\delta(x_0)\Subset\Omega$,
we have $\rho<\delta$ eventually, and
\[
 \int_{\Omega\setminus B_\delta(x_0)}|\nabla W_\tau|^2\dd x
 =\frac1{Q_\tau-\theta_R}
 \int_{\Omega\setminus B_\delta(x_0)}|\nabla G|^2\dd x
 \longrightarrow0.
\]
Together with the unit total energy, this proves concentration at
$x_0$ in the sense of \eqref{eq:concentration-definition}.

\smallskip\noindent
\emph{Step 3: a strictly positive core contribution.}
For $|z|<R$, let $b_\tau(z):=\varphi(z)+\eta(x_0+\tau z)$.
From \eqref{eq:normalized-test},
\begin{equation}\label{eq:positive-square}
 W_\tau^2(x_0+\tau z)
 =\frac{(Q_\tau+b_\tau(z))^2}{Q_\tau-\theta_R}
 =Q_\tau+2b_\tau(z)+\theta_R+
 \frac{(b_\tau(z)+\theta_R)^2}{Q_\tau-\theta_R}.
\end{equation}
The last term is nonnegative everywhere. Since
$\varphi(0)=\eta(x_0)=0$ and $\theta_R\to1/(4\pi)$,
there exist fixed $s_0,m_0>0$ such that
\begin{equation}\label{eq:core-positive-lower}
 |b_\tau(z)+\theta_R|\ge m_0\qquad(|z|\le s_0)
\end{equation}
for all sufficiently small $\tau$. These constants do not depend
on $\tau$. We may assume $R>s_0$.

Consider the core integral
\[
 J_\tau^{\mathrm{core}}:=
 \int_{B_\rho(x_0)}\!\int_{B_\rho(x_0)}
 \frac{e^{\beta_*W_\tau^2(x)}e^{\beta_*W_\tau^2(y)}}{|x-y|^\alpha}
 \dd x\dd y.
\]
Changing variables $x=x_0+\tau z$, $y=x_0+\tau w$, and using
\eqref{eq:positive-square} gives
\begin{equation}\label{eq:core-exact}
 \begin{aligned}
 J_\tau^{\mathrm{core}}
 &=S_\tau\int_{B_R}\!\int_{B_R}
 \frac{f(z)f(w)}{|z-w|^\alpha}
 e^{2\beta_*[\eta(x_0+\tau z)+\eta(x_0+\tau w)]}\\
 &\hspace{16mm}\times
 \exp\!\left(\frac{\beta_*[(b_\tau(z)+\theta_R)^2+
 (b_\tau(w)+\theta_R)^2]}{Q_\tau-\theta_R}\right)\dd z\dd w,
 \end{aligned}
\end{equation}
where $f$ is the function in \eqref{eq:trial-profile} and
\[
 S_\tau:=\tau^{4-\alpha}
 e^{2\beta_*Q_\tau+2\beta_*\theta_R}.
\]
The final exponential factor in \eqref{eq:core-exact} is at least
one on $B_R\times B_R$. By \eqref{eq:core-positive-lower} and
$e^t\ge1+t$, it is at least
$1+\beta_*m_0^2/Q_\tau$ on $B_{s_0}\times B_{s_0}$.
Set the explicit positive constant
\[
 a_0:=\beta_*m_0^2
 \mathcal B_{\alpha,\R^2}
 (f\mathbf1_{B_{s_0}},f\mathbf1_{B_{s_0}})>0.
\]
Also $\sup_{|z|\le R}|\eta(x_0+\tau z)|\le C\rho$.
Using \eqref{eq:profile-tail} in \eqref{eq:core-exact}, we find
\begin{equation}\label{eq:core-lower}
 J_\tau^{\mathrm{core}}
 \ge S_\tau(1-C\rho)
 \left(1-CR^{-(4-\alpha)/2}+\frac{a_0}{Q_\tau}\right).
\end{equation}
All constants in this estimate are independent of sufficiently
small $\tau$; they may depend on $x_0$, $\Omega$, and $\alpha$.

\smallskip\noindent
\emph{Step 4: matching the leading constant and controlling errors.}
By \eqref{eq:test-scales}, the definition of $S_\tau$, and
\eqref{eq:critical-parameters},
\[
 S_\tau=e^{2\beta_*A+2\beta_*\theta_R}
 (\kappa_\alpha+R^{-2})^{(4-\alpha)/2}.
\]
Since $\kappa_\alpha^{(4-\alpha)/2}=2\pi^2/(2-\alpha)$,
\eqref{eq:profile-energy-defect} and \eqref{eq:K-alpha} yield
\begin{equation}\label{eq:prefactor-asymptotics}
 S_\tau=K_\alpha(x_0)\bigl(1+O(R^{-2})\bigr).
\end{equation}
Our choice $R=T^2$, $\rho=e^{-T}T^2$, and $Q_\tau\sim T/(2\pi)$
gives
\begin{equation}\label{eq:test-error-orders}
 \rho=o(Q_\tau^{-1}),\quad
 R^{-2}=o(Q_\tau^{-1}),\quad
 R^{-(4-\alpha)/2}=o(Q_\tau^{-1}),\quad
 \rho^{4-\alpha}=o(Q_\tau^{-1}).
\end{equation}
For example, $R^{-(4-\alpha)/2}=T^{-(4-\alpha)}$, and
$4-\alpha>2$. Thus every negative error in
\eqref{eq:core-lower} is of smaller order than $Q_\tau^{-1}$.
Combining \eqref{eq:core-lower}, \eqref{eq:prefactor-asymptotics},
and \eqref{eq:test-error-orders}, we obtain
\begin{equation}\label{eq:core-strict-gain}
 J_\tau^{\mathrm{core}}
 \ge K_\alpha(x_0)+\frac{K_\alpha(x_0)a_0}{Q_\tau}
       +o(Q_\tau^{-1}).
\end{equation}

\smallskip\noindent
\emph{Step 5: restoring the background and identifying the limit.}
The exponential product is at least one on all of
$\Omega\times\Omega$. Hence
\[
 \begin{aligned}
 \mathcal F_{\beta_*}(W_\tau)
 &\ge I_\alpha(\Omega)+J_\tau^{\mathrm{core}}
 -\int_{B_\rho(x_0)}\!\int_{B_\rho(x_0)}|x-y|^{-\alpha}\dd x\dd y\\
 &=I_\alpha(\Omega)+J_\tau^{\mathrm{core}}
 -\rho^{4-\alpha}\mathcal B_{\alpha,B_1}(1,1).
 \end{aligned}
\]
By \eqref{eq:core-strict-gain} and \eqref{eq:test-error-orders},
this is at least
\[
 I_\alpha(\Omega)+K_\alpha(x_0)
 +\frac{K_\alpha(x_0)a_0}{2Q_\tau}
\]
for $\tau$ small. Since $Q_\tau\sim T/(2\pi)$, this proves
\eqref{eq:test-strict} after choosing $c_{x_0}>0$ sufficiently
small. In particular, the liminf is at least
$I_\alpha(\Omega)+K_\alpha(x_0)$. Step~2 proved that $W_\tau$
concentrates at $x_0$, so Proposition~\ref{prop:concentration-upper}
provides the reverse limsup inequality. This establishes
\eqref{eq:test-limit}.
\end{proof}

\subsection{Completion of the concentration formula}

\begin{proof}[Proof of Theorem~\ref{thm:concentration}]
For $x_0\in\Omega$, Proposition~\ref{prop:concentration-upper}
gives an upper bound for the supremum in
\eqref{eq:concentration-level-definition}, and the sequence in
Proposition~\ref{prop:test} gives the matching lower bound through
\eqref{eq:test-limit}. This proves \eqref{eq:concentration-formula}
in the interior and supplies the asserted strict inequality along
the test sequence.

For $x_0\in\partial\Omega$, every concentrating sequence has
functional value tending to $I_\alpha(\Omega)$ by
Proposition~\ref{prop:concentration-upper}. Such sequences exist:
choose $x_n\in\Omega$ tending to $x_0$ and radii $\rho_n\downarrow0$
with $B_{\rho_n}(x_n)\Subset\Omega$. If
$\zeta\in C_c^\infty(B_1)$ satisfies $\norm{\nabla\zeta}_2=1$,
then
\[
 u_n(x):=\zeta\bigl((x-x_n)/\rho_n\bigr)
\]
extended by zero belongs to $\mathcal S(\Omega)$, has unit energy,
and concentrates at $x_0$. This proves
$\mathcal C_\alpha(x_0)=I_\alpha(\Omega)$ on the boundary.

Finally, $r_\Omega$ is continuous on the compact set
$\overline\Omega$, is positive in $\Omega$, and vanishes on its
boundary. Thus $K_\alpha$ attains its positive maximum at an
interior point $x_m$. By \eqref{eq:test-strict}, a sufficiently
small $\tau$ gives
\[
 \Lambda_0\ge\mathcal F_{\beta_*}(W_\tau)
 >I_\alpha(\Omega)+K_\alpha(x_m)
 =\max_{x\in\overline\Omega}\mathcal C_\alpha(x).
\]
This is \eqref{eq:global-gap}.
\end{proof}

\begin{Rem}\label{rem:scope-parameters}
For each fixed $0<\alpha<2$, the test sequence approaches its
concentration level from above. This does not make it a maximizing
sequence: \eqref{eq:global-gap} puts the actual supremum strictly
higher. Also, the background splitting relies on
$\beta_*<4\pi$, so it is not uniform as $\alpha\downarrow0$.
The present argument makes no assertion about compactness of
extremals when $\alpha$ varies. Such a parameter question is
different from the fixed-parameter attainment problem; compare
\cite{WangYang2021} for a local singular setting.
\end{Rem}

\section{Attainment and the Euler--Lagrange equation}\label{sec:attainment}

\subsection{Attainment and compactness of maximizing sequences}

\begin{Prop}[Existence and strong compactness]\label{prop:attainment}
For every $0<\beta\le\beta_*$, the supremum of
$\mathcal F_\beta$ on $\mathcal S(\Omega)$ is attained.
Every maximizing sequence with unit Dirichlet norm has a subsequence
converging strongly in $H_0^1(\Omega)$ to a maximizer with unit norm.
A maximizer can be chosen nonnegative.
\end{Prop}
\begin{proof}
Consider first $\beta=\beta_*$. Take a maximizing sequence
$(u_n)\subset\mathcal S(\Omega)$ with $\norm{\nabla u_n}_2=1$;
such a choice is possible because scaling a nonzero function up
to unit norm increases the functional. Extract a subsequence as in
Corollary~\ref{cor:nonlocal-alternative}. If it concentrated at
$x_0\in\overline\Omega$, then
\[
 \Lambda_0=\lim_n\mathcal F_{\beta_*}(u_n)
 \le\mathcal C_\alpha(x_0)
 \le\max_{\overline\Omega}\mathcal C_\alpha<\Lambda_0
\]
by \eqref{eq:global-gap}, a contradiction. Thus
\eqref{eq:nonlocal-density-convergence} holds for its weak limit $u$,
and $\mathcal F_{\beta_*}(u)=\Lambda_0$.
By \eqref{eq:strict-background}, $u\ne0$. If
$\norm{\nabla u}_2<1$, the function $u/\norm{\nabla u}_2$ has
strictly larger functional value: its exponential density is no
smaller everywhere and is strictly larger on $\{u\ne0\}$, a set
of positive measure. This is impossible. Hence
$\norm{\nabla u}_2=1$. Weak convergence and equality of the limiting
Hilbert norms imply $u_n\to u$ strongly in $H_0^1(\Omega)$.
The argument applies to every normalized maximizing sequence.

If $0<\beta<\beta_*$, let $(u_n)$ be a normalized maximizing
sequence and extract a weakly and almost everywhere convergent
subsequence with limit $u$. Choose $p>q_\alpha$ such that
$p\beta<4\pi$. By \eqref{eq:mt},
$\sup_n\norm{e^{\beta u_n^2}}_p<\infty$. For every measurable
$E\subset\Omega$,
\[
 \int_E e^{q_\alpha\beta u_n^2}\dd x
 \le C|E|^{1-q_\alpha/p}.
\]
Thus the $q_\alpha$th powers are uniformly integrable. Vitali's
theorem gives $e^{\beta u_n^2}\to e^{\beta u^2}$ in
$L^{q_\alpha}(\Omega)$, and the HLS difference estimate in
Lemma~\ref{lem:transfer} gives convergence of the functional.
The supremum is strictly larger than $I_\alpha(\Omega)$, so the
same scaling argument proves $u\ne0$ and
$\norm{\nabla u}_2=1$, and yields strong convergence.
Finally, replacing a maximizer by its absolute value preserves its
Dirichlet norm and its functional value.
\end{proof}

\subsection{Differentiation at an extremal and regularity}

The next lemma supplies the differentiability needed for the
Lagrange-multiplier argument, including at the critical exponent.
Its bounds are local around a fixed Sobolev function, not uniform
on the entire unit ball.
\begin{Lem}[Differentiability of the exponential functional]
\label{lem:differentiability}
For every $a>0$ and $1\le p<\infty$, the map
$N_a:H_0^1(\Omega)\to L^p(\Omega)$,
$N_a(u)=e^{a u^2}$, is continuously Fr\'echet differentiable, with
$DN_a(u)[v]=2aue^{a u^2}v$.
Consequently $\mathcal F_a\in C^1(H_0^1(\Omega),\R)$ and
\begin{equation}\label{eq:Frechet-derivative}
 D\mathcal F_a(u)[v]=4a\int_\Omega
 \mathcal I_{\alpha,\Omega}[e^{a u^2}](x)
 u(x)e^{a u^2(x)}v(x)\dd x.
\end{equation}
\end{Lem}
\begin{proof}
Fix $u\in H_0^1(\Omega)$ and write $H(t)=e^{a t^2}$.
Then
\[
 H'(t)=2at e^{a t^2},\qquad
 H''(t)=(2a+4a^2t^2)e^{a t^2}.
\]
For fixed $m\ge0$ and $b>0$, we have
\[
 (1+|t|^m)e^{b t^2}\le C_{m,b}e^{(b+1)t^2}
 \qquad(t\in\R).
\]
Lemma~\ref{lem:fixed-function}, applied with a sufficiently large
coefficient, therefore bounds $\norm{H''(w)}_{4p}$ uniformly
for $w$ in a sufficiently small $H_0^1$ neighborhood of $u$.
Taylor's formula and H\"older's inequality therefore give, for
$\norm{\nabla v}_2$ small,
\[
 \begin{aligned}
 \norm{N_a(u+v)-N_a(u)-H'(u)v}_p
 &\le\int_0^1(1-t)
   \norm{H''(u+tv)}_{2p}\norm{v}_{4p}^2\dd t\\
 &\le C\norm{\nabla v}_2^2.
 \end{aligned}
\]
Here we used the finite-exponent Sobolev embeddings in dimension
two. Similarly, for any $z\in H_0^1(\Omega)$,
\[
 \begin{aligned}
 \norm{[H'(u+v)-H'(u)]z}_p
 &\le\int_0^1\norm{H''(u+tv)}_{4p}
       \norm{v}_{4p}\norm{z}_{2p}\dd t\\
 &\le C\norm{\nabla v}_2\norm{\nabla z}_2.
 \end{aligned}
\]
This proves differentiability and continuity of the derivative.
The continuous bilinear form in \eqref{eq:hls}, composed with
$N_a$ for $p=q_\alpha$, now gives $\mathcal F_a\in C^1$.
Symmetry of the bilinear form yields
\eqref{eq:Frechet-derivative}.
\end{proof}

\begin{Prop}[Euler--Lagrange equation and positivity]
\label{prop:EL-regularity}
Let $0<\beta\le\beta_*$ and let $u\ge0$ be a maximizer from
Proposition~\ref{prop:attainment}. Put
\[
 h:=e^{\beta u^2},\qquad
 V:=\mathcal I_{\alpha,\Omega}[h],\qquad
 \lambda:=\int_\Omega Vu^2h\dd x.
\]
Then $0<\lambda<\infty$, and $u$ satisfies
\begin{equation}\label{eq:EL}
 \begin{cases}
 -\Delta u=\lambda^{-1}Vu e^{\beta u^2}&\text{in }\Omega,\\
\quad\ \  u=0&\text{on }\partial\Omega.
 \end{cases}
\end{equation}
Moreover, $u\in W^{2,p}(\Omega)$ for every finite $p>1$,
$u\in C^{1,\gamma}(\overline\Omega)$ for every $0<\gamma<1$,
and $u>0$ in $\Omega$.
\end{Prop}
\begin{proof}
Lemma~\ref{lem:fixed-function} implies that $h$ and $u^k h$ belong
to every finite $L^p(\Omega)$ for each fixed nonnegative integer
$k$. Choose $t>2/(2-\alpha)$ and let $t'=t/(t-1)$.
Since $\alpha t'<2$, H\"older's inequality and polar integration
give
\[
 \sup_{x\in\Omega}V(x)
 \le\norm{h}_t
 \sup_{x\in\Omega}
 \left(\int_\Omega|x-y|^{-\alpha t'}\dd y\right)^{1/t'}<\infty.
\]
Thus $\lambda<\infty$. Since $u\neq 0$, $h\ge1$, and the kernel is
positive, $\lambda>0$.

By Lemma~\ref{lem:differentiability}, the usual
Lagrange-multiplier theorem on the smooth constraint
$\norm{\nabla u}_2^2=1$ applies. There exists $\ell\in\R$ such that
\[
 4\beta\int_\Omega Vu h\,v\dd x
 =2\ell\int_\Omega\nabla u\cdot\nabla v\dd x
 \qquad(v\in H_0^1(\Omega)).
\]
Testing with $v=u$ gives $\ell=2\beta\lambda$. Division yields
the weak formulation of \eqref{eq:EL}.

The right-hand side of \eqref{eq:EL} belongs to $L^p(\Omega)$
for every finite $p>1$. Global Dirichlet regularity therefore gives
$u\in W^{2,p}(\Omega)$ for all such $p$; see
\cite[Chapter~9]{GilbargTrudinger2001}. The Sobolev--Morrey
embeddings yield the stated $C^{1,\gamma}$ regularity.
Finally, $u\ge0$ is nonzero and $-\Delta u\ge0$, so the strong
maximum principle on the connected domain $\Omega$ gives $u>0$.
\end{proof}

\begin{proof}[Proof of Theorem~\ref{thm:main}]
The sharp boundedness threshold is Proposition~\ref{prop:sharp}.
Proposition~\ref{prop:attainment} proves attainment for
$0<\beta\le\beta_*$, and Proposition~\ref{prop:EL-regularity}
gives the Euler--Lagrange equation, positivity, and regularity.
\end{proof}

\par\medskip
\noindent\textbf{Acknowledgements.}\enspace
This work was supported by the Beijing University of Agriculture
Young Teachers Scientific Research Innovation Enhancement Program
under Grant No.~QJKC-2026010.



\par
\end{document}